\documentclass{amsart}
\usepackage[english]{babel}
\usepackage{amsmath}
\usepackage{amsfonts}
\usepackage{amssymb}
\usepackage{graphicx}
\usepackage{mathrsfs}
\usepackage{url}
\usepackage{multirow}
\usepackage{subfigure}
\usepackage{booktabs}

\usepackage{color}
\usepackage{subfigure}
\usepackage{longtable}
\usepackage{float}

\makeatletter
\newcommand*{\rom}[1]{\expandafter\@slowromancap\romannumeral #1@}
\makeatother

\newcommand{\st}{\mathit{s.t.}}

\newtheorem{theorem}{Theorem}[section]

\newtheorem{corollary}[theorem]{Corollary}

\theoremstyle{definition}
\newtheorem{example}[theorem]{Example}

\newtheorem{algorithm}[theorem]{Algorithm}

\numberwithin{equation}{section}
\usepackage{algpseudocode}
\newcommand{\multiline}[1]{%
    \begin{tabularx}{\dimexpr\linewidth-\ALG@thistlm}[t]{@{}X@{}}
        #1
    \end{tabularx}
}

\newcommand{\Step}[1]{\algrenewcommand{\alglinenumber}[1]{Step ##1: } #1}

\title[Global optimization for the portfolio selection model]
{Global optimization for the portfolio selection model with high-order moments}

\author[Liu Yang]{Liu~Yang}
\address{Liu Yang, Hunan Key Laboratory for Computation and Simulation in Science and Engineering, Xiangtan University; 
School of Mathematics and Computational Sciences,
Xiangtan University, Xiangtan, Hunan, China, 411105.}
\email{yangl410@xtu.edu.cn}

\author[Yi Yang]{Yi~Yang}
\address{School of Mathematics, Hunan City University, Yiyang, Hunan, China, 413000.}
\email{1623316302@qq.com}

\author[Suhan Zhong]{Suhan~Zhong}
\address{Suhan Zhong, Department of Mathematics,
Texas A\&M University, College Station, TX, USA, 77843-3368.}
\email{suzhong@tamu.edu}

\begin{document}


\date{}


\begin{abstract}
In this paper, we study the global optimality of polynomial portfolio optimization (PPO). 
The PPO is a kind of portfolio selection model with high-order moments and
flexible risk preference parameters. 
We introduce a perturbation sample average approximation method, 
which can give a robust approximation of the PPO in form of linear conic optimization. 
The approximated problem can be solved globally with Moment-SOS
relaxations. 
We summarize a semidefinite algorithm, which can be used to find reliable approximations of the optimal value and optimizer set of the PPO.
Numerical examples are given to show the efficiency of the algorithm.
\end{abstract}

\maketitle
\section{Introduction}\label{sec1}
Let $x = (x_1,\cdots,x_n)^\top$ denote the investing proportion 
(the superscript $^\top$ denotes the transpose operator).
Let $\xi = (\xi_1,\cdots, \xi_n)^{\top}$ denote the random vector
of portfolio return rate, whose distribution is supported in 
$\Omega\subseteq\mathbb{R}^n$.
The portfolio selection model with high-order moments 
(in this paper, we can also call it polynomial portfolio optimization, 
or PPO for short)
is defined as
\begin{equation}
\label{eq:def:PPO}
\left\{\begin{array}{cl}
\min & f(x) := \mathbb{E}[F(x,\xi)]\\
\st & x\in X,
\end{array}
\right.
\end{equation}
where $\mathbb{E}$ is the expectation operator, 
$F(x,\xi)$ is a polynomial loss function and 
$X\subseteq\mathbb{R}^n$ is the set of all feasible investment proportions.
Suppose short selling is not allowed, then we select
\[ X \,=\, \{ x\in\mathbb{R}^n: x\geqslant 0,\, e^{\top} x = 1 \},\]
where $e := (1,\cdots,1)^{\top}\in\mathbb{R}^{n}$ is the vector 
of all ones.
When short selling is allowed, we choose 
$X = \{ x\in\mathbb{R}^n: e^{\top} x = 1\}$. 
In this paper, we assume short selling is not allowed unless 
given extra conditions.
Consider
\begin{equation}\label{eq:polyloss}
F(x,\xi) \,=\, -\lambda_1\cdot r(x,\xi) + \lambda_2\cdot r_2(x,\xi)
+ \cdots + (-1)^d \lambda_d\cdot r_d(x,\xi),
\end{equation}
where $\lambda = (\lambda_1, \cdots, \lambda_d)^{\top}$ is a vector of
risk preference parameters such that each $\lambda_i\geqslant 0$ and 
$e^{\top}\lambda = 1$.
The $r(x,\xi) := x^{\top} \xi$ is function of portfolio return and
\begin{equation}\label{eq:center_mom}
r_i(x,\xi) \,:=\, (r(x,\xi)-\mathbb{E}[r(x,\xi)])^i
\end{equation}
for $i = 1,\cdots, d$. In the above, each $\mathbb{E}[r_i(x,\xi)]$ 
is called the $i$th {\it central moment} of $r(x,\xi)$. 
Since the central moment $\mathbb{E}[r_1(x,\xi)]$ is identically zero, 
it is ignored in the expression (\ref{eq:polyloss}).
In particular, $\mathbb{E}[r_3(x,\xi)]$ is called {\it skewness} and 
$\mathbb{E}[r_4(x,\xi)]$ is called {\it kurtosis}.
There are various portfolio loss functions. 
We would like to remark that our method does not depend on certain
polynomial parametric expressions of $F(x,\xi)$.

The portfolio selection problem is important in financial literature.
The theory of Markowitz \cite{Markowitz1952} is a fundamental work 
in this area.
It quantifies the portfolio returns and risks by the mean and variance
values of some random vector.
So the Markowitz's model is also called the mean-variance (M-V) model.
The M-V model is a popular portfolio selection model. 
However, it can be very sensitive with input noises. 
Its performance is heavily dependent on the sampling quality.
In addition, real world portfolio returns usually exhibit thick tails 
and asymmetry \cite{JonRock14ConVol,Singleton1986}.
Investors are interested in taking accounts of higher order moments 
in the decision making process \cite{HarSid00ConSke}.
Recently, higher order portfolio selection models are proposed such as 
the mean-variance-skewness (M-V-S) model \cite{Konno1995} and 
the mean-variance-skewness-kurtosis (M-V-S-K) model \cite{DaviesKat2009}.
Interestingly, these models can be seen as special cases of PPO, in a more general sense.

Apparently, the PPO (\ref{eq:def:PPO}) is a stochastic polynomial optimization problem.
Suppose the explicit parametric expression of $f$ is known, 
then (\ref{eq:def:PPO}), as a polynomial optimization problem, 
can be solved globally by Moment-SOS relaxations \cite{Lasserre2001,Laurent09}.
However, it is difficult to compute $f$ accurately in practice.
This is because computing function expectation involves
 multi-dimensional integration.
A natural approach is to apply {\it sample average approximation} 
(SAA) methods.
That is, one can approximate $f$ by the sample average function
\[
F_N(x) \,=\, \frac{1}{N}\sum\limits_{i=1}^N F(x,\xi^{(i)}),
\]
where each $\xi^{(i)}$ is a random sample following the distribution 
of $\xi$.
But it usually requires a huge size of samples to get a good approximation.
In addition, the optimizer set may still not be well approximated 
if we replace $f$ by $F_N$ in (\ref{eq:def:PPO}), 
even if the coefficient vectors of $f, F_N$ are close.
Indeed, it can be computationally expensive to extract such an 
inaccurate optimizer, before updating $F_N$ with new samples.
To address these concerns, we consider using the 
{\it perturbation sample average approximation} (PSAA) method 
proposed in \cite{NieYang2020}.
The PSAA model adds a convex regularization term to the standard SAA model.
Compared to the standard SAA model, the PSAA method gives a more robust approximation of (\ref{eq:def:PPO}).
With a proper choice of regularization parameter, 
the PSAA model can be globally solved by Moment-SOS relaxation at 
the initial relaxation order.

There is a lot of research work on portfolio optimization. 
Discussions on M-V models are given in \cite{Markow14,RubinMarkow,Stein01MarRev}.
Studies on higher order portfolio selection models are given in \cite{LevyMark79Appr,MarPar09GloptPS}.
We also refer to \cite{BlackLitt92,GoldIyen03RPS,Kolm60years} 
and references therein for other kinds of portfolio selection models.
The portfolio selection problem is a special stochastic optimization problem.
The {\it stochastic approximation} (SA) and SAA are two major 
approaches of solving stochastic optimization.
We refer to \cite{Lan20book,ShaDenRu21book} for a comprehensive introduction to stochastic optimization.

\subsection*{Contribution.}
This article studies a new portfolio selection model with high-order moments, which is given as in (\ref{eq:def:PPO}).
For an order $d\geqslant 2$, the objective function is a polynomial expectation $\mathbb{E}[F(x,\xi)]$ with
\[
F(x,\xi) = x^{\top}\xi + \sum\limits_{i = 2}^d 
(-1)^i\lambda_i \Big(x^{\top}\xi-\mathbb{E}[x^{\top}\xi]\Big)^i
\]
for some nonnegative risk preference parameters $\lambda_1,\cdots,\lambda_d$ that sum up to one.
In applications, the order $d$ and parameters $\lambda_1,\cdots,\lambda_d$
are dependent on the historic data or sample amount and investors' personal risk preferences.
In most portfolio selection models in the literature, $d$ is often selected as $d = 3,4$. To show the efficiency of our method, 
we make a numerical experiment using real world stock data and 
compute with order $d = 5$ (see as in Example~\ref{exam3}).
We give a robust approximation of (\ref{eq:def:PPO}) with perturbation sample average approximations.
This approximation can be solved globally by Moment-SOS relaxations.
We show theoretically and numerically that the optimizer set of PSAA 
model gives a good approximation of the optimizer set of original PPO.
An efficient semidefinite algorithm is proposed based on our method.
We would like to remark that our method can give a convincing
approximation of the optimizer set without any convex assumptions
on the objective function.
As far as we know, there is less work discussing about global optimality
of nonconvex portfolio selection model, prior to our work.
To sum up, our main contributions are the followings.

\begin{itemize}
\item First, we introduce a general portfolio selection model with 
high-order moments, called PPO, which is given as in (\ref{eq:def:PPO}).
This model is flexible with arbitrary moment orders and different risk appetites.
\item Second, we give a robust approximation of PPO with perturbation
sample average approximations. The PSAA model is more computationally efficient compared to the standard SAA model.
For proper samples and parameters, the optimizer set of PSAA is close to
that of the original PPO.
\item Third, we propose a Moment-SOS algorithm to solve the PSAA model 
globally. Numerical experiments are given to show the efficiency of our algorithm.
\end{itemize}

The rest of the paper is organized as follows. 
In Section~\ref{sec:pre}, we review main results of polynomial optimization.  
In Section~\ref{sec:PolyPortOpt}, we introduce a PSAA model of the PPO and propose a semidefinite algorithm. 
Numerical experiments are presented in Section~\ref{sc:numerical}. 
Conclusions are given in Section~\ref{sec:con}.

\section{Preliminaries}\label{sec:pre}
\subsection*{Notation}
The symbol $\mathbb{N}$ denotes the set of nonnegative integers, 
and $\mathbb{R}$ denotes the set of real numbers.
For $t\in\mathbb{R}$, $\lceil t\rceil$ denotes the smallest integer 
that is bigger than or equal to $t$.
The $\mathbb{R}^n$ is the $n$-dimensional Euclidean space, and 
$\|\cdot\|$ stands for the Euclidean norm.
For a point $u$ and a scalar $R>0$, denote 
$B(u,R) := \{x: \|x-u\|\leqslant R\}$.
We use $e = (1,\cdots,1)^{\top}\in\mathbb{R}^n$ to denote the vector 
of all ones. The $e_i\in\mathbb{R}^n$ denotes the vector of all zeros
except the $i$th entry being one.
Let $x = (x_1,\cdots,x_n)^{\top}$.
Denote by $\mathbb{R}[x]$ the real polynomial ring in $x$, 
and denote by $\mathbb{R}[x]_d$ the set of polynomials with degrees no more than $d$.
For a polynomial $f$, we use $\deg(f)$ to denote its degree. 
For a tuple of polynomials $g = (g_1,\cdots, g_m)^{\top}$, 
we denote $\deg(g) := \max\{\deg(g_1),\cdots,\deg(g_m) \}$.
A symmetric matrix $A\in \mathbb{R}^{n\times n}$ is said to be positive
semidefinite or psd, if $x^{\top}Ax\ geqslant 0$ for all  $x\in \mathbb{R}^n$.
We use $A\succeq 0$ to denote a matrix $A$ is psd.
For two sets $A, B\subseteq\mathbb{R}^n$, denote the distance
\begin{equation}\label{eq:distance} 
\mathbb{D}(A, B) \,:=\, \sup_{u\in A}\inf_{v\in B}\|u-v\|. 
\end{equation}
For a given power vector 
$\alpha = (\alpha_1,\cdots,\alpha_n)^{\top}\in\mathbb{N}^n$, 
we denote the monomial
\[ x^{\alpha} := x_1^{\alpha_1}\cdots x_n^{\alpha_n}, \]
where the degree $\vert\alpha\vert := \alpha_1 + \cdots + \alpha_n$.
For a degree $d$, denote 
$\mathbb{N}_d^n := \{ \alpha\in \mathbb{N}^n: \vert\alpha\vert \leqslant d \}$.
The symbol $[x]_d$ denotes the monomial vector of $x$ with the highest degree $d$, i.e.,
\[ [x]_d \,:=\,
(1,\, x_1,\, \cdots,\, x_n,\, x_1^2,\,x_1x_2,\, \cdots,\, x_n^d)^{\top}. \]

\subsection{Polynomial optimization}\label{sc:PO}
A polynomial $ \sigma\in \mathbb{R}[x] $ is said to be sum-of-squares
(SOS) if it can be written as $\sigma=p_{1}^{2}+\cdots+p_{k}^{2}$ 
for some $p_1,\cdots,p_k\in\mathbb{R}[x]$.
We denote by $\Sigma[x]$ the set of all SOS polynomials and denote 
$\Sigma[x]_d:=\Sigma[x] \cap \mathbb{R}[x]_{d}$ for every degree $d\in \mathbb{N}$.

Let $g=(g_1,\cdots,g_m)$ be a tuple of real polynomials in $x$.
It determines a semi-algebraic set
\begin{equation}\label{eq:K:semialg}
K \,=\, \{ x\in\mathbb{R}^n: g(x)\geqslant 0 \}.
\end{equation}
Denote the cone of nonnegative polynomials on $K$ by
\[ \mathscr{P}(K) \,:=\, \{ p\in \mathbb{R}[x]: 
p(x)\geqslant 0\,(\forall x\in K) \}. \]
For a degree $d$, denote $\mathscr{P}_d(K):=\mathscr{P}(K)\cap\mathbb{R}[x]_d$. The quadratic module of $g$ is
\[ \mbox{QM}(g) \,=\, \Sigma[x]+g_{1} \cdot \Sigma[x]+\cdots+g_{m} \cdot \Sigma[x]. \]
For $k\geqslant \lceil \deg(g)/2\rceil$, we denote the $k$th order truncation
\begin{equation}\label{eq:QMtrun}
\mbox{QM}(g)_k \,:=\, \Sigma[x]_{2 k}+g_{1} \cdot \Sigma[x]_{2k-\deg(g_1)}+\cdots+g_{m} \cdot \Sigma[x]_{2k-\deg(g_m)}.
\end{equation}
The $\mbox{QM}(g)$ is a convex cone, and each $\mbox{QM}(g)_k$ is a 
convex cone in $\mathbb{R}[x]_{2k}$.
For every $k$, the containment relations hold that
\[ \mbox{QM}(g)_k\subseteq \mathscr{P}_{2k}(K),\quad 
\text{and}\quad 
\mbox{QM}(g)_k\subseteq \mbox{QM}(g)_{k+1}\subseteq \mbox{QM}(g). \]
The quadratic module $\mbox{QM}(g)$ is said to be {\it archimedean} 
if there is $p\in \mbox{QM}(g)$ such that $p\geqslant 0$ determines a compact set.
Suppose $\mbox{QM}(g)$ is archimedean, then $K$ must be compact.
Conversely, it may not be true. 
However, given $K$ is compact (i.e., $K \subseteq B(0, R)$ for some radius $R$), 
one can always set $\bar{g} := (g, R^2-\|x\|^2)$ such that $\bar{g}\geqslant 0$ 
also determines $K$ and $\mbox{QM}(\bar{g})$ is archimedean.
It is clear that every $p\in \mbox{QM}(g)$ is nonnegative on $K$.
If a polynomial $q>0$ over $K$, then $q\in \mbox{QM}(g)$ under the archimedean condition. 
This conclusion is often referenced as Putinar's Postivstellensatz \cite{Putinar93}.
Interestingly, under some optimality conditions, 
if $f\geqslant 0$ on $K$, we also have $f\in \mbox{QM}(g)$.
This result is shown in \cite{Nie14fincov}.

Consider a polynomial optimization problem
\begin{equation}
\label{eq:genpo}
\left\{
\begin{array}{cl}
\min\limits_{x\in\mathbb{R}^n} & p(x)\\
\st & g_1(x)\geqslant 0,\cdots, g_m(x)\geqslant 0,
\end{array}
\right.
\end{equation}
where $p$ and each $g_i$ are polynomials.
Denote $g = (g_1,\cdots, g_m)$ and let $K$ be given as in (\ref{eq:K:semialg}).
Suppose $\gamma^*$ is the global optimal value of (\ref{eq:genpo}).
Then $\gamma\leqslant \gamma^*$ if and only if $p(x)-\gamma\geqslant 0$ for every $x\in K$, i.e., $p(x)-\gamma\in\mathscr{P}(K)$.
This implies a hierarchy of SOS relaxations of (\ref{eq:genpo}).
For a degree $k\geqslant \max\{ \lceil \deg(p)/2\rceil ,\lceil \deg(g)/2\rceil\}$, the $k$th order SOS relaxation of (\ref{eq:genpo}) is
\[
\left\{
\begin{array}{rl}
\gamma_k:= \max\limits_{\gamma\in\mathbb{R}} & \gamma\\
\st & p(x) - \gamma\in\mbox{QM}(g)_k.
\end{array}
\right.
\]
Its dual problem is the $k$th order moment relaxation.
Under some archimedean conditions, it is shown in \cite{Lasserre2001} 
that $\gamma_k\rightarrow \gamma^*$ as $k\rightarrow\infty$.
In particular, the finite convergence is studied in \cite{Nie13CerConv,Nie14fincov}.
The polynomial optimization has broad applications, see \cite{FanNieZhou18,GNY21,NieTang21CGNEP,NYZZ21DRO}.
We refer to \cite{HNY2021,QuTang22,NieYang2021} for recent results of polynomial optimization.
For a detailed review for polynomial optimization, 
we refer to monographs \cite{NieBook23,HenKorLas20,Las15book} and references therein.

\subsection{Localizing and moment matrices}
Denote by $\mathbb{R}^{\mathbb{N}_{d}^{n}}$ the real vector space 
that consists of truncated multi-sequences (tms) 
$y=\left(y_{\alpha}\right)_{\alpha \in \mathbb{N}_{d}^{n}}$ of degree $d$.
Each tms $y$ defines a Riesz functional $\mathscr{R}_{y} $ on $ \mathbb{R}[x]_{d} $ by
\[ \mathscr{R}_{y} (f)
\,:=\, \sum_{\alpha \in \mathbb{N}_{d}^{n}} f_{\alpha} y_{\alpha}
\quad \mbox{if}\quad f(x) = \sum_{\alpha \in \mathbb{N}_{d}^{n}} f_{\alpha} x^{\alpha}. \]
For convenience, we denote that
\begin{equation}\label{eq:<f,y>}
\langle f, y\rangle \,:=\, \mathscr{R}_{y}(f).
\end{equation}
We say a tms $y\in\mathbb{R}^{\mathbb{N}_d^n}$ admits a measure 
$\mu$ supported in $K$ if it satisfies $y_{\alpha} = \int x^{\alpha}\mathtt{d}\mu$ for all $\alpha\in\mathbb{N}_d^n$.
A special case is that $y$ admits the Dirac measure at the point 
$u = (u_1,\cdots, u_n)^{\top}$.
Then for every power vector $\alpha = (\alpha_1,\cdots, \alpha_n)^T$, we have
\[ y_{\alpha} = u_1^{\alpha_1}\cdots u_n^{\alpha_n}. \]
In particular, it holds that
\[ u \,= \, (y_{e_1},\, \cdots,\, y_{e_n})^{\top}. \]
The problem for checking whether or not a tms admits a measure is called the {\it truncated moment problem}. 
We refer to \cite{CurFial05KTMP,HelNie12KTMP,Nie14ATMP} for more details on this topic.

Let $q\in\mathbb{R}[x]$ and  $t_0 = \lceil \deg(q)/2\rceil$. 
For $k\geqslant t_0$ and a tms $y\in\mathbb{R}^{\mathbb{N}_{2k}^n}$, 
there exists a symmetric matrix $L_q^{(k)}[y]$ such that
\begin{equation}\label{eq:locmat}
vec(a)^{T}(L_{q}^{(k)}[y])vec(b) \,=\,
\mathscr{R}_{y}(q a b),\quad \forall a,b\in\mathbb{R}[x]_{k-t_0},
\end{equation}
where $vec(a), vec(b)$ denote the coefficient vectors of $a,b$ respectively.
The $L_q^{(k)}[y]$ is called the $k$th order {\it localizing matrix} of $q$. 
For example, when $n=3, k=2$ and $q=x_1^2+2x_2x_3$, we have
\[
L_q^{(2)}[y] = \begin{bmatrix}
y_{200}+2y_{011} & y_{300}+2y_{111} & y_{210}+2y_{021} & y_{201}+2y_{012}\\
y_{300}+2y_{111} & y_{400}+2y_{211} & y_{310}+2y_{121} & y_{301}+2y_{112}\\
y_{210}+2y_{021} & y_{310}+2y_{121} & y_{220}+2y_{031} & y_{211}+2y_{022}\\
y_{201}+2y_{012} & y_{301}+2y_{112} & y_{211}+2y_{022} & y_{202}+2y_{013}
\end{bmatrix}.
\]
For the special case that $q=1$, we define the $k$th order moment matrix by
\begin{equation}
	\label{eq:mommat}
	M_{k}[y] \,:=\, L_{1}^{(k)}[y].
\end{equation}
For example, when $n=2, k=1$, we have
\[
M_1[y] = \begin{bmatrix}
y_{00} & y_{10} & y_{01}\\
y_{10} & y_{20} & y_{11}\\
y_{01} & y_{11} & y_{02}
\end{bmatrix}.
\]
The moment and localizing matrices are useful in polynomial optimization. 
For a tuple of polynomials $g = (g_1,\cdots,g_m)$ and a degree $k\geqslant \lceil \deg(g)/2\rceil$, define
\begin{equation} \label{mom:S(g):2d}
	\mathscr{S}(g)_{k} :=
	\left\{
	y \in \mathbb{R}^{ \mathbb{N}^n_{2k} } :
	M_k[y] \succeq 0, \,  L_{g_1}^{(k)}[y] \succeq 0, \cdots,
	L_{g_m}^{(k)}[y] \succeq 0
	\right\}.
\end{equation}
The $\mathscr{S}(g)_k$ is a convex cone and it is the dual cone of $\mbox{QM}(g)_k$, i.e.,
\[
(\mbox{QM}(g)_k)^* = \mathscr{S}(g)_k,
\]
where the superscript $^*$ denotes the dual cone.

\section{Global portfolio optimization}
\label{sec:PolyPortOpt}

For a portfolio consisting of $n$ assets, let $\xi\in\Omega\subseteq\mathbb{R}^n$ be the random rate of portfolio return and $x\in\mathbb{R}^n$ be the investing proportion.
Assume short selling is not allowed, then the PPO model is
\begin{equation}
	\label{eq:PolyPS}
	\left\{
	\begin{array}{cl}
		\min\limits_{x\in \mathbb{R}^n} & f(x)= \mathbb{E}[F(x,\xi)] \\
		\st & x\geqslant 0,\, e^{\top}x=1,
	\end{array}
	\right.
\end{equation}
where $F(x,\xi)$ is given as in (\ref{eq:polyloss}).
Note the constraint $x\geqslant 0$ can be ignored if short selling is allowed.
For both cases, the constraining equation $e^{\top}x=1$ can be used to eliminate one variable, i.e.,
\[
e^{\top}x = 1\quad \Leftrightarrow\quad x_n = 1-(x_1+\cdots+x_{n-1}).
\]
Let $\bar{x} := (x_1,\cdots,x_{n-1})^{\top}$ and $\bar{e} := (1,\cdots,1)^{\top}\in\mathbb{R}^{n-1}$.
We define
\begin{equation}
	\label{eq:f1}
	f_0(\bar{x}) \,:=\, f(\bar{x}, 1-\bar{e}^{\top}\bar{x}) 
	\,=\, \mathbb{E}[F(\bar{x},1-\bar{e}^{top}\bar{x}, \xi)].
\end{equation}
It is clear that the PPO (\ref{eq:PolyPS}) is equivalent to
\begin{equation}\label{eq:eliPolyPS}
\left\{
\begin{array}{rl}
\theta := \min\limits_{\bar{x}\in\mathbb{R}^{n-1}} & f_0(\bar{x})\\
\st\,\,\, & \bar{x}\geqslant 0,\, 1-\bar{e}^{\top}\bar{x}\geqslant 0.
\end{array}
\right.
\end{equation}
Similarly, when the short selling is allowed, the PPO is equivalent 
to the unconstrained optimization
\begin{equation}\label{eq:eliPSuc}
\min\, f_0(\bar{x})\quad \st \quad \bar{x}\in\mathbb{R}^{n-1}.
\end{equation}
The polynomial $f$ is bounded from below on the feasible set of (\ref{eq:PolyPS}), which is compact.
A feasible point $u = (u_1,\cdots, u_{n-1})^{\top}$ is a global minimizer of (\ref{eq:eliPolyPS}) 
if and only if $x^* = (u_1,\cdots, u_{n-1}, 1-\bar{e}^{\top}u)^{\top}$ is the global minimizer of (\ref{eq:PolyPS}).
Suppose (\ref{eq:eliPSuc}) is solvable with an optimizer, 
then a similar conclusion holds for the PPO with short selling allowed.
Both (\ref{eq:eliPolyPS})-(\ref{eq:eliPSuc}) have fewer decision variables than the original PPO problems.
So solving (\ref{eq:eliPolyPS})-(\ref{eq:eliPSuc}) saves computational expenses.
In the following discussions, we focus on these reformulations (\ref{eq:eliPolyPS})-(\ref{eq:eliPSuc}).

\subsection{PSAA model}
In applications, the distribution of $\xi$ is typically unknown.
But it is natural to assume the truncated moment information of $\xi$ 
can be well approximated by samplings and historic data.
So we consider applying SAA methods to approximate $f_0(\bar{x})$ as in (\ref{eq:f1}).

Suppose $\xi^{(1)},\cdots, \xi^{(N)}$ are given samples that each identically follows the distribution of $\xi$.
Define the sample average function
\begin{equation}\label{eq:dehSAAf}
f_N(\bar{x}) \,:=\, \frac{1}{N}\sum\limits_{i=1}^N 
F(\bar{x},1-\bar{e}^{\top}\bar{x},\xi^{(i)}).
\end{equation}
The construction of $f_N(\bar{x})$ guarantees that $f_0(\bar{x}) = \mathbb{E}[f_N(\bar{x})]$.
If we further assume that every $\xi^{(i)}$ is independently identically distributed, then under some regularity conditions, we have
\[
f_N(\bar{x})\rightarrow f_0(\bar{x})\quad 
\mbox{as}\quad N\rightarrow\infty
\]
with probability one.
This result is implied from the Law of Large Numbers \cite{rossbook}.
Then we can approximate (\ref{eq:eliPolyPS}) by the following optimization
\begin{equation}\label{eq:SAApoly}
\left\{\begin{array}{rl}
\theta_N := \min\limits_{\bar{x}\in\mathbb{R}^{n-1}} & f_N(\bar{x}) \\
\st\,\,\,\, & \bar{x}\geqslant 0,\, 1-\bar{e}^{\top}\bar{x}\geqslant 0.
\end{array}
\right.
\end{equation}
The above approximation is usually given from classic SAA methods.
It is a deterministic polynomial optimization, 
which can be solved globally by Moment-SOS relaxations, see in Subsection~\ref{sc:PO}.
It has good statistical properties. 
By \cite[Theorem~5.2]{ShaDenRu21book} and \cite[Theorem~5.3]{ShaDenRu21book}, 
we have the following convergent result.
\begin{theorem}\label{theo:SAA}
Suppose $\theta, \theta_N$ are the optimal values, and $T, T_N$ are the
nonempty optimizer sets of (\ref{eq:eliPolyPS}) and (\ref{eq:SAApoly}) respectively.
Assume the distribution of $\xi$ has a compact support $\Omega\subseteq \mathbb{R}^n$, 
and $\{\xi^{(i)}\}_{i=1}^N$ are independent identically distributed samples of $\xi$.
Then $\theta_N\rightarrow \theta $ and the distance $\mathbb{D}(T, T_N)\rightarrow 0$ with probability one, as $N\rightarrow\infty$.
\end{theorem}
\begin{proof}
Let $\Delta$ denote the feasible set of (\ref{eq:eliPolyPS}) and (\ref{eq:SAApoly}). Clearly, $\Delta$ is compact.
The $f_0(\bar{x})$ and each $f_N(\bar{x})$ are polynomials in $\bar{x}$, which are continuous in $\Delta$.
So both (\ref{eq:eliPolyPS}) and (\ref{eq:SAApoly}) are solvable with nonempty optimizer sets $T,T_N\subseteq \Delta$.
Assume the distribution of $\xi$ is supported in a compact set $\Omega$.
For every fixed $u\in \Delta$, the function $F(u,1-\bar{e}^{\top}u,\xi)$ as a polynomial in $\xi$, is continuous and integrable over $\Omega$.
Therefore, the assumptions of \cite[Theorem 5.2]{ShaDenRu21book} and \cite[Theorem 5.3]{ShaDenRu21book} are all satisfied.
The conclusions can be directly implied by these two theorems.
\end{proof}
The conclusions of Theorem~\ref{theo:SAA} depend on the compactness 
of the feasible set.
Suppose short selling is allowed, the PPO is equivalent to an unconstrained optimization, see as in (\ref{eq:eliPSuc}).
In this case, it can be hard to ensure the global optimizer set of the SAA model is close to that of the original PPO.
Indeed, it is not necessary to solve each SAA model very accurately, even for cases that short selling is not allowed.
Because the performance of SAA model depends on the sampling quality and quantity, 
and to extract the global minimizer of (\ref{eq:SAApoly}) may require a high relaxation order.
One can imagine that the computational expenses becomes very high if we need to solve the SAA model (\ref{eq:SAApoly}) every time when a new sample is generated.
On the other hand, investors are interested in the best investing proportions, equivalently, the optimizer sets of PPO reformulations (\ref{eq:eliPolyPS})-(\ref{eq:eliPSuc}).
So our goal is to get convincing approximations of global optimizers of PPO problems in an efficient way.

To achieve the goal, we consider using the perturbation sample average approximation method introduced in \cite{NieYang2020}.
It adds a convex regularization term to the standard SAA model, 
which has some good computational properties. Denote
\begin{equation}	
d_0 \,:=\, \lceil \deg(f_N)/2\rceil.
\end{equation}
The PSAA model of (\ref{eq:eliPolyPS}) is
\begin{equation}\label{eq:eliPolyPSAA}
\left\{\begin{array}{cl}
\min\limits_{\bar{x}\in\mathbb{R}^{n-1}} & f_N(\bar{x})+\varepsilon \|[\bar{x}]_{2d_0}\| \\
\st & \bar{x}\geqslant 0,\, 1-\bar{e}^{\top}\bar{x}\geqslant 0.
\end{array}\right.
\end{equation}
In the above, $f_N$ is given as in (\ref{eq:dehSAAf}), $\varepsilon>0$ is a small parameter, and $[\bar{x}]_{2d_0}$ is the monomial vector of $\bar{x}$ with the highest degree $2d$, i.e.,
\[
[\bar{x}]_{2d_0} \,=\, (1,\, x_1,\, \cdots,\, x_{n-1},\, x_1^2,\, x_1x_2,\,\cdots,\, x_{n-1}^{2d_0})^{\top}.
\]
Similarly, the PSAA model for (\ref{eq:eliPSuc}) is
\begin{equation}\label{eq:eliPolyPSAA1}
\min\, f_N(\bar{x})+\varepsilon \|[\bar{x}]_{2d_0}\|\quad \st\quad \bar{x}\in \mathbb{R}^{n-1}.
\end{equation}
It is worth noting that when $\varepsilon = 0$, 
(\ref{eq:eliPolyPSAA})--(\ref{eq:eliPolyPSAA1}) are reduced to be standard SAA models.
Suppose these standard SAA models are solvable with nonempty global minimizer sets, say $T_N^{(1)}, T_N^{(2)}$.
Then for each $\varepsilon>0$, the PSAA models (\ref{eq:eliPolyPSAA})--(\ref{eq:eliPolyPSAA1}) must also be solvable with optimizer sets, 
say $\hat{T}_N^{(1)}, \hat{T}_N^{(2)}$ respectively.
In addition, $\mathbb{D}(T_N^{(1)}, \hat{T}_N^{(1)})$ and  $\mathbb{D}(T_N^{(2)}, \hat{T}_N^{(2)})$ are both close to zero when $\varepsilon$ is sufficiently small.
It implies that solving PSAA models can also return reliable approximations for the original solution set of PPO, 
if $T_N^{(i)}, i=1,2$ are good approximations.

We would like to remark that PSAA models (\ref{eq:eliPolyPSAA})--(\ref{eq:eliPolyPSAA1}) are more robust than SAA models from the aspect
of computational efficiency.
The optimization problems (\ref{eq:eliPolyPSAA})--(\ref{eq:eliPolyPSAA1}) 
can be solved globally by Moment-SOS relaxations.
For the special case that $\varepsilon = 0$, the relaxation order may be high to return a global minimizer.
But for a proper choice of $\varepsilon>0$, (\ref{eq:eliPolyPSAA})--(\ref{eq:eliPolyPSAA1}) can be solved globally 
with a unique minimizer, 
at the initial relaxation order.
More details on this are given in the following subsection.

\subsection{Moment-SOS relaxations}\label{ssc:msr}
In this section, we study Moment-SOS relaxations for PSAA models.
Denote the tuple of constraining polynomials
\begin{equation}\label{eq:g}
g(\bar{x}) \,:=\,  (x_1,\, \cdots,\,  x_{n-1},\,  1-\bar{e}^{\top}\bar{x}).
\end{equation}
Suppose a tms $y\in \mathbb{R}^{\mathbb{N}_{2d_0}^{n-1}}$ 
satisfies $y=[\bar{x}]_{2d_0}$ for some point $\bar{x}\in\mathbb{R}^{n-1}$, then
\[
f_N(\bar{x})+\varepsilon \|\bar{x}\|_{2d_0} \,=\, 
\langle f_N,y\rangle+\varepsilon\|y\|,
\]
where $\langle,\rangle$ is given as in (\ref{eq:<f,y>}).
If $\bar{x}$ is a feasible point for (\ref{eq:eliPolyPSAA}), making $g(\bar{x})\geqslant 0$, then we have
\[ p(\bar{x}) = \langle p,y\rangle\geqslant 0,\quad
\forall p\in \mbox{QM}(g)_{d_0}. \]
This follows from the defining equation of $\mbox{QM}(g)_{d_0}$ as in (\ref{eq:QMtrun}). It implies that
\[ y \,\in\, \mathscr{S}(g)_{d_0} \,=\, (\mbox{QM}(g)_{d_0})^*, \]
where $\mathscr{S}(g)_{d_0}$ is the tms cone given as in (\ref{mom:S(g):2d}).
In other words, we have
\[ M_{d_0}[y] \succeq 0,\quad L_{g_i}^{(d_0)}[y]\succeq 0, \]
for $i = 1,\ldots,n$. The $M_{d_0}[y]$ and each $L_{g_i}^{(d_0)}[y]$ are the $d_0$th order moment and localizing matrices of $y$ and $g_i$.
It leads to the moment relaxation of (\ref{eq:eliPolyPSAA}):
\begin{equation}\label{eq:relax_pro_PSAA}
\left\{
\begin{array}{cl}
\min &\langle f_N, y\rangle+\varepsilon \| y\|  \\
\st &  y_{0}=1,\,M_{d_0}[y]\succeq 0,\\
& L_{g_i}^{(d_0)}[y]\succeq 0,\, i=1,\cdots,n,\\
& y\in \mathbb{R}^{\mathbb{N}_{2d_0}^{n-1}}.
\end{array}\right.
\end{equation}
Its dual problem is the SOS relaxation
\begin{equation}\label{eq:dual_relax_pro}
\left\{
\begin{array}{cl}
\max & \gamma \\
\st &  f_N(\bar{x})-q(\bar{x})-\gamma \in \mbox{QM}(g)_{d_0},\\
& \|vec(q)\|\leqslant \varepsilon,\\
& q\in \mathbb{R}[\bar{x}]_{2d_0},\, \gamma\in\mathbb{R},
\end{array}\right.
\end{equation}
where $vec(q)$ denotes the coefficient vector of the polynomial $q$.
The conic constraint $f_N(\bar{x})-q(\bar{x})-\gamma\in \mbox{QM}(g)_{d_0}$ means that 
$f_N(\bar{x})-q(\bar{x})-\gamma$, as a polynomial in $\bar{x}$, belongs to the truncated quadratic module $\mbox{QM}(g)_{d_0}$.
Similarly, we can get the moment relaxation of (\ref{eq:eliPolyPSAA1}):
\begin{equation}\label{eq:relax_pr_PSAA_uc}
\left\{\begin{array}{cl}
\min & \langle f_N, y\rangle+\varepsilon\|y\|\\
\st & y_0 = 1,\, M_{d_0}[y]\succeq 0,\\
& y\in \mathbb{R}^{\mathbb{N}_{2d_0}^{n-1}}.
\end{array}\right.
\end{equation}
The SOS relaxation of (\ref{eq:eliPolyPSAA1}) is
\begin{equation}\label{eq:relax_du_PSAA_uc}
\left\{\begin{array}{cl}
\max & \gamma \\
\st &  f_N(\bar{x})-q(\bar{x})-\gamma \in \Sigma[\bar{x}]_{2d_0},\\
& \|vec(q)\|\leqslant \varepsilon,\\
& q\in \mathbb{R}[\bar{x}]_{2d_0},\, \gamma\in\mathbb{R}.
\end{array}\right.
\end{equation}
The optimization problems (\ref{eq:relax_pro_PSAA})--(\ref{eq:relax_du_PSAA_uc}) 
are linear conic optimization problem with cones given by linear, second-order and semidefinite constraints.

For an optimization problem, we say its relaxation is {\it tight} if 
the relaxation has the same optimal value of the original problem.
Then we give a sufficient and necessary condition such that 
(\ref{eq:relax_pro_PSAA}) is a tight relaxation of (\ref{eq:eliPolyPSAA}).
\begin{theorem}\label{theo:PSAA}
For a given $\varepsilon>0$, suppose (\ref{eq:relax_pro_PSAA}) is solvable with a global minimizer $y^*$. 
Then (\ref{eq:relax_pro_PSAA}) is a tight relaxation of (\ref{eq:eliPolyPSAA}) if and only if 
$\mbox{rank}\, M_{d_0} [y^{*}] = 1$. 
Moreover, suppose $\mbox{rank}~ M_{d_0} [y^{*}]=1$, then the 
$u = (y^*_{e_1}, \cdots, y^*_{e_{n-1}})^{\top}$ is a global minimizer of (\ref{eq:eliPolyPSAA}).
\end{theorem}
\begin{proof}
Let $\varphi_1, \varphi_2$ be optimal values of (\ref{eq:eliPolyPSAA}) and (\ref{eq:relax_pro_PSAA}) respectively.
Suppose $u^*$ is the global minimizer of (\ref{eq:eliPolyPSAA}).
Since $[u^*]_{2d_0}$ is feasible for (\ref{eq:relax_pro_PSAA}),
we have
\[ \varphi_1 \,=\, \langle f_N, [u^*]_{2d_0}\rangle +
\varepsilon\|[u^*]_{2d_0}\|\,\geqslant\, \varphi_2. \]
Suppose $y^*$ is the global minimizer of (\ref{eq:relax_pro_PSAA}) such that $\mbox{rank}\, M_{d_0}[y^*] = 1$.
Then $M_{d_0}[y^*] = [u]_{d_0}([u]_{d_0})^T$ for $u = (y^*_{e_1}, \cdots, y^*_{e_{n-1}})^{\top}$. 
It implies that $y^* = [u]_{2d_0}$, and each $g_i(u)\geqslant 0$.
This is because $L_{g_i}^{(d_0)}[y^*]\succeq 0$ and that $g_i(u)$ equals the $(1,1)$th-entry of $L_{g_i}^{(d_0)}[y^*]$.
Then the $u$ is feasible for (\ref{eq:eliPolyPSAA}) and
\[ \varphi_1 \,\leqslant\, f_N(u)+\varepsilon\|[u]_{2d_0}\| 
\,=\, \langle f_N, y^*\rangle+\varepsilon\|y^*\| \,=\, \varphi_2. \]
So $\varphi_1 = \varphi_2$ and (\ref{eq:relax_pro_PSAA}) is a tight relaxation of (\ref{eq:eliPolyPSAA}). 
Furthermore, we have $u=(y_{e_1}^*,\cdots, y_{e_{n-1}}^*)^{\top}$ is the global minimizer of (\ref{eq:eliPolyPSAA}).
	
Assume that $\varphi_1 = \varphi_2$, then $[u^*]_{2d_0}$ is a global minimizer of (\ref{eq:relax_pro_PSAA}).
Note (\ref{eq:relax_pro_PSAA}) has a closed convex feasible set and a strictly convex objective function.
So $y^* = [u^*]_{2d_0}$ is the unique global minimizer of (\ref{eq:relax_pro_PSAA}) and we have 
$M_{d_0}[y^*] = [u^*]_{d_0}([u^*]_{d_0})^{\top},$ which implies $ \mbox{rank} M_{d_0}[y^*] = 1$.
\end{proof}
\begin{corollary}\label{coro:PSAA-unconstrained}
For a given $\varepsilon>0$, suppose (\ref{eq:relax_pr_PSAA_uc}) is solvable with a global minimizer $\hat{y}^*$. 
Then (\ref{eq:relax_pr_PSAA_uc}) is a tight relaxation of (\ref{eq:eliPolyPSAA1}) 
if and only if $\mbox{rank}\, M_{d_0} [\hat{y}^{*}] = 1$. 
Moreover, suppose $\mbox{rank}~ M_{d_0} [\hat{y}^{*}]=1$,
then the $u = (\hat{y}^*_{e_1}, \cdots, \hat{y}^*_{e_{n-1}})^{\top}$ is a global minimizer of (\ref{eq:eliPolyPSAA1}).
\end{corollary}

\subsection{A semidefinite algorithm}
We propose a semidefinite algorithm to solve the perturbation sample average approximation of the portfolio selection model.
When the short selling is not allowed, the PSAA model is given as in (\ref{eq:eliPolyPSAA}).
When the short selling is allowed, the PSAA model is given as in (\ref{eq:eliPolyPSAA1}).

\begin{algorithm}\label{alg:PSAA}
For the given sample points $\xi^{(1)}, \ldots, \xi^{(N)}$ and a small parameter $\varepsilon >0$, do the following:
\begin{algorithmic}[1]

\Step \State Compute the sample average function $f_N(\bar{x})$ given as in (\ref{eq:dehSAAf}). 
If short selling is not allowed, go to Step~2. Otherwise, go to Step~3.
		
\Step \State	Solve dual pairs (\ref{eq:relax_pro_PSAA})--(\ref{eq:dual_relax_pro}).
If (\ref{eq:dual_relax_pro}) is infeasible, increase the value of the parameter $\varepsilon $ (i.e., update $\varepsilon = 2\varepsilon$), until (\ref{eq:relax_pro_PSAA}) has a global optimizer $y^{*} $. 
Then go to Step~4.
		
\Step \State Solve dual pairs (\ref{eq:relax_pr_PSAA_uc})--(\ref{eq:relax_du_PSAA_uc}).
If (\ref{eq:relax_du_PSAA_uc}) is infeasible, increase the value of the parameter $\varepsilon $ (i.e., update $\varepsilon = 2\varepsilon$), until (\ref{eq:relax_pr_PSAA_uc}) has a global optimizer $y^{*} $. 
Then go to Step~4.
		
\Step \State Let $u = (y^*_{e_{1}}, y^*_{e_{2}}, \cdots, y^*_{e_{n-1}})^{\top}$.
Output candidate solution $x^*$ of (\ref{eq:def:PPO})
\[ x^* \,=\, (u^T,\, 1- \bar{e}^{\top}u)^{\top}, \]
and terminate the algorithm.
\end{algorithmic}
\end{algorithm}

In Algorithm~\ref{alg:PSAA}, every optimization problem can be solved efficiently by \texttt{MATLAB} software 
\texttt{GloptiPoly3} \cite{HenrionLasserre2009} and \texttt{SeDuMi} \cite{Sturm1999}.
In Step~2, there always exists $\varepsilon>0$ that is big enough such that the optimization pair (\ref{eq:relax_pro_PSAA})--(\ref{eq:dual_relax_pro}) is solvable.
For (\ref{eq:relax_pro_PSAA}), it is strictly feasible since 
$\{\bar{x}\in \mathbb{R}^{n-1}: g(\bar{x})>0\}$ is nonempty.
On the other hand, when $\varepsilon>0$ is big, i.e., 
$\varepsilon>\|vec(f_N-[\bar{x}]_{d_0}^{\top}[\bar{x}]_{d_0})\|$, 
we have (\ref{eq:dual_relax_pro}) also be strictly feasible.
Then there is a strong duality between  (\ref{eq:relax_pro_PSAA})--(\ref{eq:dual_relax_pro}) 
and that each relaxation problem is solvable with a minimizer.
Similar arguments can be applied in Step~3.
When $\varepsilon>0$ is big enough, 
the dual pair  (\ref{eq:relax_pr_PSAA_uc})--(\ref{eq:relax_du_PSAA_uc})
is always solvable.
We refer to \cite[Theorem~3.3]{NieYang2020} for more details.
In Step~4, the output point $x^*$ is feasible for the original PPO.
Suppose the short selling is not allowed.
Since $y^*$ is feasible for (\ref{eq:dual_relax_pro}), 
we have $L_{g_i}^{(d_0)}[y^*]\succeq 0$ for each $i$,
which implies
\[\begin{array}{c}
y_{e_i}^* = \Big(L_{x_i}^{(d_0)}[y^*]\Big)_{1,1}\geqslant 0,\quad\mbox{and}\\
1 - (y_{e_1}^*+\cdots +y_{e_{n-1}}^*) = 
\Big(L_{1-\bar{e}^T\bar{x}}^{(d_0)}[y^*]\Big)_{1,1}\geqslant 0,
\end{array}\]
where $(\cdot)_{1,1}$ refers to the $(1,1)$th entry of the given matrix.
In other words, we have $u\geqslant 0,\, 1-\bar{e}^{\top}u\geqslant 0$, 
so $x^*$ is feasible for (\ref{eq:def:PPO}).
When the short selling is allowed, the $x^*$ is always feasible since 
$\bar{e}^{\top}u+(1-\bar{e}^{\top}u)=1$.

In previous discussions, we have shown that when $\varepsilon>0$ is big enough, 
the moment relaxations (\ref{eq:relax_pro_PSAA}) and (\ref{eq:relax_pr_PSAA_uc}) is solvable with a global optimizer.
However, the optimizer set of the PSAA model may be far away from the optimizer set of the original PPO, if $\varepsilon$ has a big value.
To get a good approximation of the original optimizer set, 
a preferable value of $\varepsilon$ can be $\varepsilon=\varepsilon_1$ such that for every $0\leqslant \varepsilon < \varepsilon_1$, 
the moment relaxations are not solvable.
Such a bound can be determined as follows.
Suppose the short selling is not allowed, 
we solve
\begin{equation}\label{eq:besteps}
\left\{\begin{array}{rl}
\varepsilon_1 = \min & \|vec(q)\|\\
\st &  f_N(\bar{x}) - q(\bar{x}) -\gamma \in \mbox{QM}(g)_{d_0},\\
& q\in \mathbb{R}[\bar{x}]_{2d_0},\, \gamma\in\mathbb{R}.
\end{array}\right.
\end{equation}
When the short selling is allowed, we solve
\begin{equation}\label{eq:e}
\left\{\begin{array}{rl}
\varepsilon_1 = \min & \|vec(q)\|\\
\st &  f_N(\bar{x}) - q(\bar{x}) -\gamma \in \Sigma[\bar{x}]_{2d_0},\\
& q\in \mathbb{R}[\bar{x}]_{2d_0},\, \gamma\in\mathbb{R}.
\end{array}\right.
\end{equation}
These are linear conic optimization problems with linear, second-order cone and semidefinite constraints.
In applications, it is computational expensive to determine $\varepsilon_1$ for each $f_N$.
So we usually select some small heuristic values of $\varepsilon$, i.e., 
$\varepsilon = 0.01$ or $\varepsilon = 0.001$.

\section{Numerical Experiments}
\label{sc:numerical}
In this section, we give some numerical examples to show the efficiency of Algorithm \ref{alg:PSAA}.
The computation is implemented in \texttt{MATLAB} R2014a, in a computer with CPU 6th Generation Intel\textregistered Core\textregistered i5-6500 CPU and RAM 8 GB.
The \texttt{MATLAB} software \texttt{GloptiPoly3} \cite{HenrionLasserre2009} and \texttt{SeDuMi} \cite{Sturm1999} are used to solve PSAA models.
For neatness of the paper, we only display four decimal digits to show computational results.
In each example, for samples $\xi^{(1)},\cdots, \xi^{(N)}$, we denote sample average functions
\begin{equation}
\label{eq:avgr}
r_{i,N}(\bar{x}) \,:=\, \left\{\begin{array}{ll}
\frac{1}{N}\sum\limits_{i=1}^N r\big(\bar{x}, 1-\bar{e}^{\top}\bar{x},\xi^{(i)}\big), & \mbox{if $i = 1$},\\
\frac{1}{N}\sum\limits_{i=1}^N r_i\big(\bar{x},1-\bar{e}^{\top}\bar{x},\xi^{(i)}\big), & \mbox{if $i = 2,\cdots, d$}, 
\end{array}
\right.
\end{equation}
where $r(x,\xi) = x^{\top}\xi$ is the portfolio return and each 
$r_i(x,\xi)$ is given as in (\ref{eq:center_mom}).

\begin{example}
\label{exm:1comp_SAA_PSAA}
Suppose there are three stocks $S_1$, $S_2$ and $S_3$ in the market, 
which allows for short selling.
Let $\xi=(\xi_1, \xi_2, \xi_3)^{\top}$ denote the random return rates of $S_1,S_2,S_3$ respectively.
Assume $\xi_1,\xi_2,\xi_3$ are independent distributed and each follows a normal distribution, i.e.,
\[ \xi_1\sim \mathcal{N}(0.92,1.8),\quad 
\xi_2\sim \mathcal{N}(0.64,1.2),\quad 
\xi_3\sim \mathcal{N}(0.41,1.4). \]
In the above, $\mathcal{N}(\mu,\sigma)$ stands for a normal distribution with the mean value $\mu$ and the standard deviation $\sigma$.
Consider the M-V-S model induced from (\ref{eq:def:PPO})--(\ref{eq:polyloss}) with the order $d = 3$ and the risk preference vector
\[ \lambda \,=\, ( 0.2,\, 0.5,\, 0.3 )^T. \]
We compute the analytic object function of the M-V-S model
\[\begin{array}{l}
f(x) \,=\,  \mathbb{E}[-0.2 r(x,\xi) + 0.5 r_2(x,\xi)-0.3 r_{3}(x,\xi)]\\
\quad \quad \,\,\,=\, -0.184x_1 - 0.128 x_2 - 0.082x_3 + 0.9x_1^2 + 0.6x_2^2 + 0.7x_3^2.\\
\end{array}\]
Since the short selling is allowed, the feasible set is 
\[ X \,=\, \{x\in\mathbb{R}^3: x_1+x_2+x_3 = 1\}, \] 
thus $1-x_1-x_2$ is the investment proportion of stock $S_3$. 
We can equivalently reformulate the PPO into
\begin{equation}\label{eq:e-m-v-s}
\left\{\begin{array}{cl}
\min & 0.618-1.502x_1-1.446x_2+1.6x_1^2+1.4x_1x_2+1.3x_2^2\\
\st & \bar{x} = (x_1,x_2)\in\mathbb{R}^2.
\end{array}\right.
\end{equation}
It can be solved globally with Moment-SOS relaxations.
The optimal value and solution of (\ref{eq:e-m-v-s})  are
\[ f_ {min} = 0.1089,\quad \bar{x}^* = ( 0.2957, 0.3969 )^{\top}. \]
It implies that the best investment proportion is
\[ x_{min} \,=\, \begin{pmatrix}\bar{x}^*\\ 1-\bar{e}^{\top}\bar{x}^*\end{pmatrix} \,=\,
( 0.2957, 0.3969, 0.3074)^{\top}.\]
We make a sampling of $\xi$ with the sample size $N=1000$ and get
\[\begin{array}{l}
r_{1,N}(\bar{x}) \,=\, 0.4216+0.4423x_1+0.2728x_2,\\
r_{2,N}(\bar{x}) \,=\, 1.4253-2.8306x_1-2.7282x_2+3.1689x_1^2+
2.6621x_1x_2+2.4906x_2^2,\\
r_{3,N}(\bar{x}) \,=\, 0.0535-0.0975x_1+0.2576x_2+0.2974x_1^2-0.6222x_1x_2-0.4961x_2^2
\\
\qquad\qquad -0.1243x_1^3-0.2888x_1^2x_2+0.6121x_1x_2^2+0.3540x_2^3.
\end{array}
\]
For the PSAA model (\ref{eq:eliPolyPSAA1}), it has the objective function $f_N(\bar{x})+\varepsilon \Vert[\bar{x}]_{4}\Vert$, with
\begin{equation}\label{eq:m-v-s}
f_N(\bar{x}) \,=\, -0.2r_{1,N}(\bar{x})+0.5r_{2,N}(\bar{x})-
0.3r_{3,N}(\bar{x}).
\end{equation}
Apply Algorithm~\ref{alg:PSAA} to this PPO. 
We report all numerical results in Table~\ref{tab:Results-diff}.
The $\varepsilon$ is the regularization parameter. 
The ``Solvable?'' refers to the status of (\ref{eq:relax_pro_PSAA}).
The ``Time'' stands for the total CPU time of running Algorithm~\ref{alg:PSAA} with the unit second.
The $x^* = (\bar{x}^{\top}, 1-\bar{e}^{\top}u)^{\top}$ denotes the candidate solution of (\ref{eq:def:PPO}) solved from Algorithm~\ref{alg:PSAA}.
The $f_N(u)$ gives an approximation for the optimal value of (\ref{eq:def:PPO}).	
\begin{table}[htb]
\centering
\caption{Numerical results for Example~\ref{exm:1comp_SAA_PSAA}}
\label{tab:Results-diff}
\begin{tabular}{lclll}
\toprule
$\varepsilon$ & Solvable? & Time & $x^* $ & $f_N(u)$\\
\midrule
$0$ & No &  $0.4796$ & not available & not available\\
$0.006$ & Yes & $0.3018$ & $(0.2833,    0.4134,    0.3033)$ &   $0.0995$\\
$0.01$ & Yes &  $0.1037$ & $(0.2833,    0.4124,    0.3043)$ & $0.0995$\\
\bottomrule
\end{tabular}
\end{table}
From Table \ref{tab:Results-diff}, 
we can see that the PSAA model (\ref{eq:m-v-s}) performs better for this problem, compared to the unperturbed case (i.e, $\varepsilon=0$).
It gives reliable optimizers, while the classical SAA model (i.e, $\varepsilon=0$) does not return good ones.
\end{example}

\begin{example}\label{exm:2comp_SAA_PSAA}
Assume that the short selling is not allowed.
Let $S_1, S_2, S_3$ be three stocks, of which the random vector of return each follows a normal distribution. 
We report their expected returns (mean) and risks (covariance) in Table~\ref{tab:Data_Stocks}.
\begin{table}[htb]
\centering
\caption{The expectation and covariance of the random vector in Example~\ref{exm:2comp_SAA_PSAA}}
\begin{tabular}{cccrr}
\toprule 
\multirow{2}*{Stocks name} & \multirow{2}*{Expected return} & \multicolumn{3}{c}{Covariance}   \\
\cmidrule{3-5}
&& $S_1$& $S_2$ & $S_3$\\
\midrule
$S_1$ & $0.91$ & $1.90$ &$0.38$ &$1.20$\\
$S_2$ & $0.65$ &$0.38$ & $1.50$ &$-0.80$  \\
$S_3$ & $0.49$ &$1.20$ & $-0.80$ & $1.70$\\
\bottomrule
\end{tabular}
\label{tab:Data_Stocks}
\end{table}	
Consider the M-V model induced from (\ref{eq:def:PPO})--(\ref{eq:polyloss}) with the order $d = 2$ and the risk preference vector
\[  \lambda \,=\, (0.75,\,0.25)^{\top}. \]
With the given data in Table \ref{tab:Data_Stocks}, we get the analytic object function of the M-V model
\[\begin{array}{l}
f(x) \,=\, \mathbb{E}[-0.75 r(x,\xi) + 0.25 r_2(x,\xi)]\\
 \quad \quad \,\,\,=\, -0.6825x_1-0.4875x_2-0.3675x_3+0.475x_1^2+0.19x_1x_2+0.6x_1x_3\\
\quad\quad\quad\quad +0.375x_2^2-0.4x_2x_3+0.425x_3^2.
\end{array}
\]
By eliminating $x_3$ with the equality constraint $e^{\top}x = 1$, 
the PPO can be equivalently reformulated into
\begin{equation}
\label{eq:m-v}
\left\{
\begin{array}{cl}
\min\limits_{\bar{x}\in\mathbb{R}^2} & 0.0575-0.565x_1-1.37x_2+0.3x_1^2+0.84x_1x_2+1.2x_2^2\\
\mbox{s.t.}& x_{1}\geqslant 0,\, x_{2}\geqslant 0,\, 1-x_1-x_2\geqslant 0.
\end{array}
\right.
\end{equation}
The $\bar{x}=(x_{1}, x_{2})^T $ represents the investment proportion of the stocks $S_1$ and $S_2$, 
thus $1-x_1-x_2$ represents the investment proportion of the stock $S_3$. 
The polynomial optimziation problem (\ref{eq:m-v}) can be solved globally by Moment-SOS relaxations.
The optimal value and solution are
\[ f_ {min} = -0.3455,\quad \bar{x}^* = (0.2794, 0.4730 )^{\top}.\]
It implies that the best investment proportion is 
\[ x_{min} \,=\, \begin{pmatrix}\bar{x}^*\\ 1-\bar{e}^{\top}\bar{x}^*\end{pmatrix} \,=\, (0.2794,\, 0.4730,\, 0.2476 )^{\top}.\]
Then we generate random samples of $\xi$ with different sizes.
The sample averages for expected value and covariance of these samplings are reported in Table~\ref{tab:mean_var_1000}.
\begin{table}[h]
\centering
\caption{The sample mean and covariance of Example \ref{exm:2comp_SAA_PSAA}}\label{tab:mean_var_1000}
\begin{tabular}{cccccc}
\toprule
\multirow{2}*{Sample size}&\multirow{2}*{Stock name}&\multirow{2}*{Sample mean }& \multicolumn{3}{c}{Sample covariance}  \\
\cmidrule{4-6}
&&&$S_1$&$S_2$&$S_3$\\
\midrule 
&$S_1$ &  $0.8650$  & $1.8961$ & $-0.0156$  &  $0.0509$ \\
$N=1000$& $S_2$  & $0.6952$ & $-0.0156$ &   $1.4958$ &   $0.0829$ \\
&$S_3$  & $0.5399$ &  $0.0509$  & $ 0.0829 $ &  $1.7187$\\
\midrule
&$S_1$ &  $1.3768$  & $2.1085$ & $-0.3918$  &  $-0.0580$ \\
$N=100$& $S_2$  &$0.7064$ &$-0.3918$ &   $1.5020$ &  $-0.2511$ \\
&$S_3$  & $0.7678$ & $-0.0580$  & $-0.2511$  &  $1.8118$\\
\bottomrule
\end{tabular}
\end{table}
Apply Algorithm~\ref{alg:PSAA}. 
We report all numerical results in Table \ref{tab:value-diff-varepsilon-1000}.
The $x^*=(\bar{x}^{\top}, 1-\bar{e}^{\top}u)^{\top}$ denotes the candidate solution of (\ref{eq:def:PPO}) solved from Algorithm~\ref{alg:PSAA}.
The $f_N(u)$ gives an approximation for the optimal value of (\ref{eq:def:PPO}).
\begin{table}[h]
\centering
\caption{Numerical results for Example~\ref{exm:2comp_SAA_PSAA}}
\label{tab:value-diff-varepsilon-1000}
\begin{tabular}{cccccc}
\toprule
Sample size & $\varepsilon$ & $x^* $ & $f_N(\bar{x}^*)$ & $f_N(\bar{x}^*)-f_{min}$& Time\\
\midrule
\multirow{3}*{ $N=1000$}& $0$ & $(0.4412, 0.3859, 0.1729)$ & $-0.3933$ & $0.0478$ & $0.2495$\\
&$0.001$& $(0.4410, 0.3857, 0.1734 )$ & $-0.3933$ & $0.0478$ & $0.2234$\\
&$0.01$ & $( 0.4389, 0.3839, 0.1772)$ & $-0.3933$ & $0.0478$ & $0.2523$\\
\midrule
\multirow{3}*{ $N=1000$}& $0$ & $( 0.5811, 0.2531, 0.1657)$ &  $ -0.6520$ &  $0.3065$& $0.2099$ \\
&$0.001$ &  $( 0.5807,  0.2532,  0.1661)$ &   $-0.6520$ &  $0.3065$ & $0.2406$ \\
&$0.01$ &  $( 0.5770, 0.2534, 0.1696)$ &  $-0.6520$ &  $0.3065$ & $0.2649$ \\
\bottomrule
\end{tabular}
\end{table}
From Table~\ref{tab:value-diff-varepsilon-1000},
when short selling is not allowed, 
it is clear that the SAA and PSAA have close solutions when the parameter $\varepsilon$ is small.
It shows that once SAA gives a good approximation of the original PPO, 
then so as PSAA with a proper parameter.
\end{example}

In the following example, we study the real stock data from the Chinese stock market.
A high-order portfolio model is constructed in form of (\ref{eq:def:PPO}).
We apply Algorithm~\ref{alg:PSAA} in this PPO problem to see its performance.

\begin{example}\label{exam3}
We select 4 stocks from the Chinese stock market for portfolio investment, which are Linhai Stock ($S_1$), Shanghai Automotive ($S_2$), 
Hangzhou Iron and Steel ($S_3$), and CYTS ($S_4$).
We collected weekly stock closing prices from March 3, 2006 to March 25, 2021 in the Choice software database. 
For $t= 1, \cdots, N+1$ and for each $i=1, \cdots, 4$, 
we use $P_{i}^{(t)}$ to denote the price of $S_i$ in week $t$.
Then the stock return ~$\xi_{i}^{(t)}$ of $S_i$ in week~$t$~ can be computed as
\[ \xi_{i}^{(t)}=(P_{i}^{(t+1)}-P_{i}^{(t)})/P_{i}^{(t)}. \]
The weekly return rates of these four stocks from March 3, 2006 to 
March 25, 2021 are ploted in the following figures.
The subfigures (a)-(d) shows the weekly return for $S_1$-$S_4$ respectively.
\begin{figure}[h]
\centering
\subfigure[ Linhai Stock]{
\includegraphics [scale=0.35] {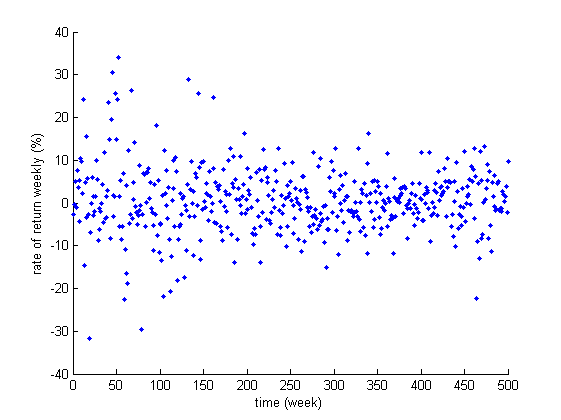} \label{fig:linhai}
}
\hfill
\subfigure[Shanghai Automotive]{
\includegraphics [scale=0.35] {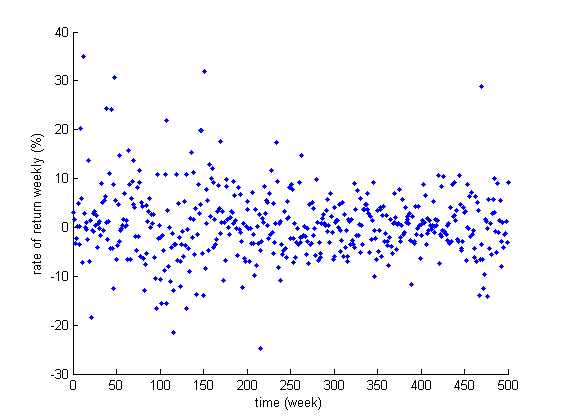} \label{fig:Shanghai}
}
\hfill
\subfigure[Hangzhou Iron and Steel]{
\includegraphics[scale=0.35]{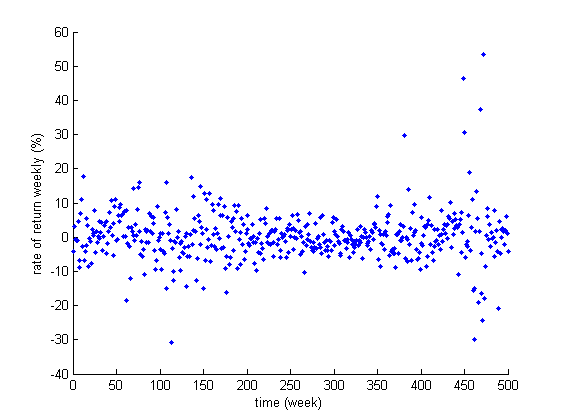}\label{fig:Hangzhou}
}
\hfill
\subfigure[CYTS]{
\includegraphics [scale=0.35]{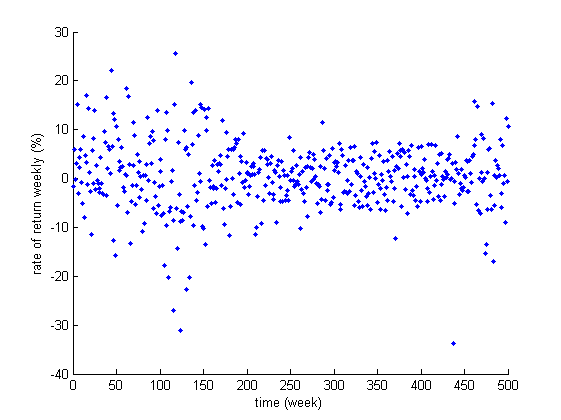}\label{fig:CYTS}
}
\caption{Scatter plot of the weekly return of the four stocks ($N$=500)}
\label{fig:stock}
\end{figure}
It is easy to observe that the distribution of each stock returns is not normally distributed.
They tend to be asymmetric, leptokurtic and heavy-tailed.
For instance, the kurtosis in Figure \ref{fig:stock} is of a sharp peaks and fat tail character.
Therefore, the M-V model, which only uses the first and the second order moments, may not be efficient to give a promising investing plan.
This is because many good properties of M-V model relies on the assumption that the stock returns follow a normal distribution.
On the other hand, higher moments are important in measuring this probability distribution of the stock returns.
For instance, the skewness can be used to measure the skew direction and degree of statistical data distribution. It also represents the asymmetric characteristics of statistical data.
In financial literature, if the skewness is positive, it means that the positive returns are easy to generate.
If the skewness is negative, it means that the potential risk is greater than the potential profit.
Investors prefer the portfolio with a large skewness  and dislike the yield with a large kurtosis.

Based on previous analysis, we consider a PPO model given as in (\ref{eq:def:PPO})--(\ref{eq:polyloss}) with $d=5$.
Assume the short selling is not allowed.

\noindent
(i)\, We choose the risk preference vector as
\[
\lambda \,=\, (0.205,\, 0.213,\, 0.221,\, 0.200,\,  0.161)^{\top}.
\]
With the sample size $N=720$ and the parameter $\varepsilon=0.01$, 
the PSAA model is
\begin{equation}
\label{eq:m-v-s-k-f}
\left\{
\begin{array}{cl}
\min\limits_{\bar{x}\in\mathbb{R}^3} & f_N(\bar{x})+\varepsilon \Vert[\bar{x}]_{6}\Vert\\
\st &  x_1\geqslant 0, x_2\geqslant 0, x_3\geqslant 0,\\
& 1-x_1-x_2-x_3\geqslant 0.
\end{array}
\right.
\end{equation}
In the above, $x_1,x_2,x_3$ denote the investment proportion of the stocks Linhai Stock, Shanghai Automotive, Hangzhou Iron and Steel respectively.
Then $1-x_1-x_2-x_3$ will be the investment proportion of the stock CYTS.
Apply Algorithm~\ref{alg:PSAA}. 
We get the candidate solution for the original PPO, 
and an approximation for its optimal value:
\[
x^* \,=\, (0.2517,\, 0.2560,\, 0.2447,\, 0.2436)^{\top},
\quad f_N(\bar{x}^*) \,=\, 0.0110.
\]

\noindent
(ii)\, Next, we consider different parameter vector $\lambda=(\lambda_1, \lambda_2, \lambda_3, \lambda_4 , \lambda_5)^{\top}$ to compare numerical results of  different investors' preferences.
In financial literature, the larger $\lambda_i~(i=2,4)$ imply investors are the more risk-averse.
Conversely, the smaller $ \lambda_i~(i=2,4)$ are, 
the more risk-appetite investors are.
In particular, when $\lambda_2+\lambda_4=1$, investors are extremely risk-averse and their only goal is to minimize the risk of their portfolio.
When $\lambda_2+\lambda_4=0$, investors are extremely fond of risk, 
and portfolio returns are the only factors that affect decision making.
In our numerical experiments, we take three scenarios as follows.
\[
\begin{array}{ll}
\mbox{Scenario 1:}& \lambda=(0.2070,  0.2060 , 0.2020 , 0.2050 , 0.1800)^{\top}.\\
\mbox{Scenario 2:}& \lambda=(0.0005 ,  0.8300 , 0.0005 , 0.1385  , 0.0205)^{\top}.\\
\mbox{Scenario 3:}& \lambda=(0.5270 , 0.0305, 0.2020 , 0.0305 , 0.2100)^{\top}.\\
\end{array}
\]
The scenario 1, scenario 2, scenario 3 respectively reflect that the investor is risk neutral, risk aversion, and risk appetite.
Apply Algorithm~\ref{alg:PSAA} in these three cases. 
We report all numerical results in Table \ref{tab:MVSK_diff_lambda}.
The $x^* = (\bar{x}^{\top},1-\bar{e}^{\top}\bar{x}^*)^{\top}$ denotes the candidate solution of (\ref{eq:def:PPO}) solved from Algorithm~\ref{alg:PSAA}.
The $f_N(\bar{x}^*)$ gives an approximation for the optimal value of (\ref{eq:def:PPO}).
\begin{table}[h]
\centering
\caption{Numerical results for Example~\ref{exam3}}
\label{tab:MVSK_diff_lambda}
\begin{tabular}{ccc}
\toprule
Scenario & $x^*$ &   $f_N(\bar{x}^*)$\\
\midrule
1 & $(0.2517, 0.2600, 0.2447, 0.2436)$ &  $0.0110$ \\
2 & $(0.2522, 0.2527, 0.2393, 0.2559)$ & $0.0127$   \\
3 & $(0.2526, 0.2742,0.2440, 0.2293)$ &  $0.0095$  \\
\bottomrule
\end{tabular}
\end{table}

To better explain that the different risk attitude leads to different investing plan, 
we selected the weekly return rates of the four stocks in the first 20 weeks. 
The volatility of the return rate can be observed from Figure \ref{fig:stock1}.
\begin{figure}[h]
\centering
\includegraphics [scale=0.30] {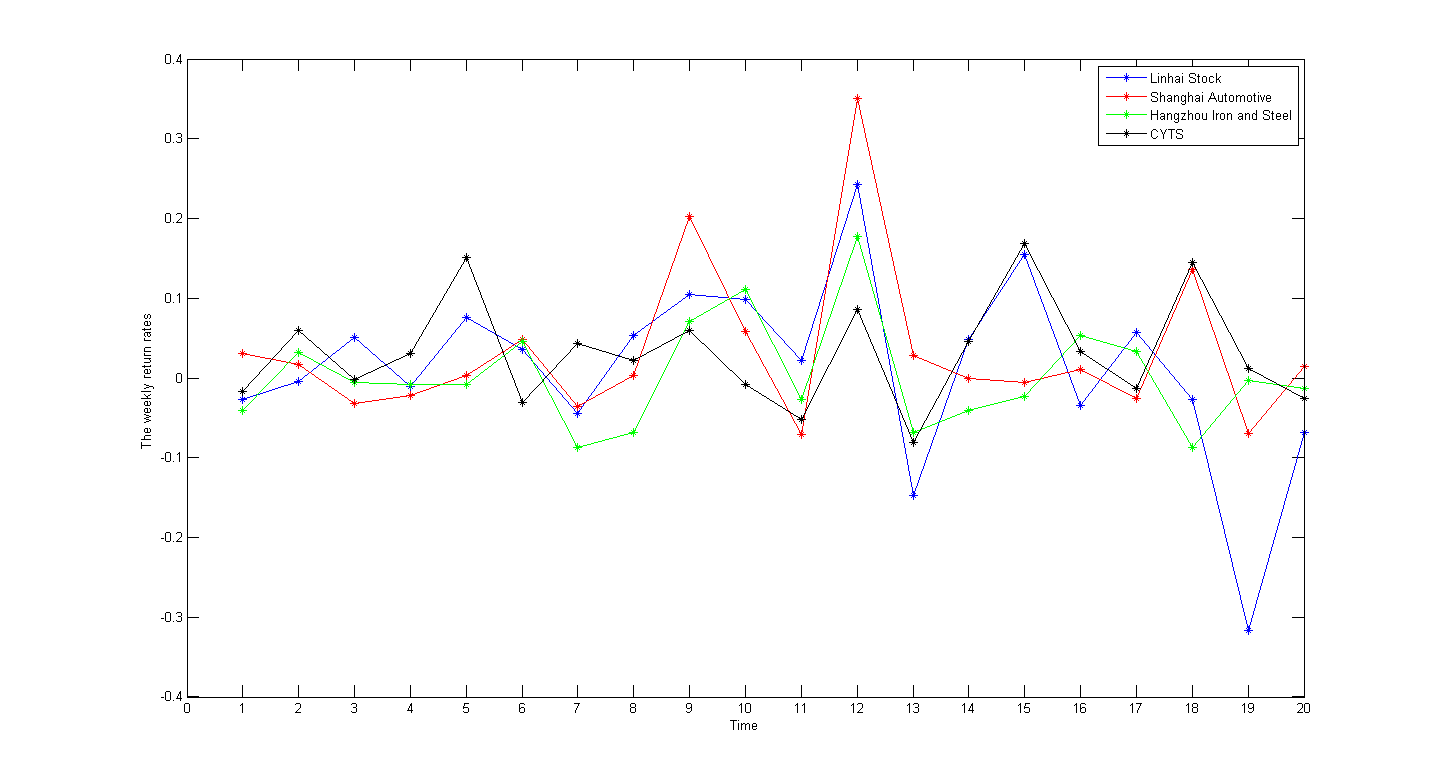}
\caption{The weekly return rates of the four stocks}
\label{fig:stock1}
\end{figure}
From Figure \ref{fig:stock1} and Table \ref{tab:MVSK_diff_lambda}, 
we can see that risk appetite investors invest more in Shanghai Automotive
 stock than risk neutral investors and risk averse investors.
This is because the return rate of Shanghai Automotive stock is more volatile than the others.
Risk averse investors invest more in CYTS stock than risk neutral investors and risk appetite investors, because the return rate of CYTS stock is stable and the risk is low.
\end{example}

\section{Conclusions}\label{sec:con}
In this paper, we study a portfolio selection model with high-order moments.
A polynomial portfolio optimization model is proposed, 
which can overcome the deficiencies of the M-V model and cater to investors with different risk appetites.
To solve the PPO, we introduce perturbation sample average approximations. It gives a robust approximation of PPO and can be efficiently solved by Moment-SOS relaxations.
We summarize a semidefinite algorithm for solving PSAA globally. 
The algorithm can be used to find reliable approximations of optimal value and optimizer set of the original PPO.
Numerical experiments are given to show the efficiency of our methods.

\end{document}